\documentclass[11pt]{amsart}

\title{Constant mean curvature cylinders with irregular ends}

\author{M.~Kilian}
\address{M. Kilian, Department of Mathematics,
University College Cork, Ireland.}
\email{m.kilian@ucc.ie}

\author{N.~Schmitt}
\address{N. Schmitt, Mathematisches Institut,
Universit\"at T\"ubingen, Germany}
\email{nschmitt@mathematik.uni-tuebingen.de}


\usepackage{url}
\usepackage{bbm}
\usepackage{graphicx}

\numberwithin{equation}{section}

\theoremstyle{plain}
\newtheorem{theorem}{Theorem}[section]
\newtheorem*{theorem*}{Theorem}

\newtheorem{proposition}[theorem]{Proposition}
\newtheorem{lemma}[theorem]{Lemma}

\newcommand{\abs}[1]{{\lvert#1\rvert}}

\newcommand{\half}{\tfrac{1}{2}}

\newcommand{\fourth}{\tfrac{1}{4}}

\newcommand{\ol}[1]{\overline{#1}}

\DeclareMathSymbol{\varnothing}{\mathord}{AMSb}{"3F}

\DeclareMathOperator{\Order}{O}

\DeclareMathOperator{\diag}{diag}

\DeclareMathOperator{\id}{\mathbbm{1}}

\DeclareMathOperator{\trace}{tr}

\newcommand{\bbC}{\mathbb{C}}

\newcommand{\bbR}{\mathbb{R}}
\newcommand{\bbS}{\mathbb{S}}

\newcommand{\calC}{\mathcal{C}}
\newcommand{\calD}{\mathcal{D}}

\newcommand{\CPone}{\bbC\mathrm{P}^1}

\newcommand{\AND}{\quad\text{and}\quad}

\newcommand{\MatrixGroup}[1]{{\mathrm{#1}}}

\newcommand{\matSL}[2]{\MatrixGroup{SL}_{#1}{#2}}

\newcommand{\matSU}[2]{\MatrixGroup{SU}_{#1}{#2}}

\newcommand{\matsl}[2]{\MatrixGroup{sl}_{#1}{#2}}
\newcommand{\matsu}[2]{\MatrixGroup{su}_{#1}{#2}}

\newcommand{\gauge}[2]{{{#1}{.}{#2}}}

\newcommand{\spaceperiod}{\,\,\,.}
\newcommand{\spacecomma}{\,\,\,,}
\newcommand{\coloneq}{=}
\newcommand{\eqcolon}{=}
\newcommand{\deriv}{\mathrm{d}}

\newcommand{\Cstar}{{\bbC^\times}}

\newcommand{\scale}{c}

\begin{document}


\begin{abstract}
We prove the existence of a new class of constant mean curvature cylinders with an arbitrary number of umbilics by unitarizing the monodromy of Hill's equation.
\end{abstract}

\maketitle

\section*{Introduction}

In this paper we prove the existence of a new class of constant mean curvature ({\sc{cmc}}) cylinders with an arbitrary number of umbilics. The construction is based on the generalized Weierstrass representation~\cite{Dorfmeister_Pedit_Wu_1998}. Families of {\sc{cmc}} cylinders with umbilics in Euclidean 3-space were first found in~\cite{Kilian_2000,Kilian-McIntosh-Schmitt-2000, Dorfmeister-Haak-2003}.
The method involves solving a holomorphic complex linear
$2 \times 2$ system of ordinary differential equations with values in a loop group. The solution is subsequently projected to a moving frame of the Gauss map, from which the associated family can be constructed. The input for this procedure is a {\emph{potential}}, which in the case of cylinders lives on a Riemann sphere with two punctures.

The two obstacles to proving the existence of cylinders are the unitarity of the monodromy, and the closing conditions. While either one of these conditions can be encoded in the potential, the other one is then not immediate. Thus one can either ensure unitarity by working with skew-hermitian potentials ~\cite{Kilian_2000, Kilian-McIntosh-Schmitt-2000, Dorfmeister-Kobayashi-2007}, or one can work with potentials which guarantee the closing conditions but do not automatically have unitary monodromy~\cite{Dorfmeister-Haak-2003, Kilian_Kobayashi_Rossmann_Schmitt_2005}. This article takes the second approach, and presents a very general class of potentials which naturally encode the closing conditions, and for which there exists a very simple unitarizer. The general form of our potentials is the same as those
of $k$-noids~\cite{Rossman_Schmitt_2006, Schmitt_Kilian_Kobayashi_Rossman_2007} with Delaunay ends,
and higher genus surface with ends~\cite{Kilian_Kobayashi_Rossmann_Schmitt_2005}, and thus fits nicely into a general framework of potentials for non-compact {\sc{cmc}} surfaces.

The $2 \times 2$ first order system of ordinary differential equations is equivalent to a second order equation. For cylinders the most general form is equivalent to Hill's equation. The two singular points (usually taken to be $z=0$ and $z=\infty$) in Hill's equation correspond to the two ends of the resulting cylinder. A regular singular point gives a Delaunay end \cite{Kilian_Rossman_Schmitt_2008}. The potential for our new family of cylinders can be viewed as a superposition of a Delaunay potential with two potentials, each of which superimposes an irregular singularity. The resulting cylinders, shown in Figure~\ref{fig:cylinder}, appear as two Smyth-like irregular ends attached to a Delaunay cylinder.

Our cylinder potentials are parameterized by a holomorphic
function $f:\Cstar \to \bbC$. If $f$ is holomorphic at $z=0$, then that end is asymptotic to a Delaunay surface~\cite{Kilian_Rossman_Schmitt_2008, Kilian-McIntosh-Schmitt-2000}. To ensure that the monodromy is conjugate to an element of the unitary loop group we impose symmetries on the function $f$. In particular we show that it suffices to prescribe reality conditions along the real line and the unit circle to ensure the existence of a diagonal unitarizer. Since all known existence proofs of cylinders require the imposition of some symmetries, it would be interesting to extend our construction to cylinders without symmetries.


\begin{figure}[ht]
\centering
\includegraphics[width=0.32\textwidth]{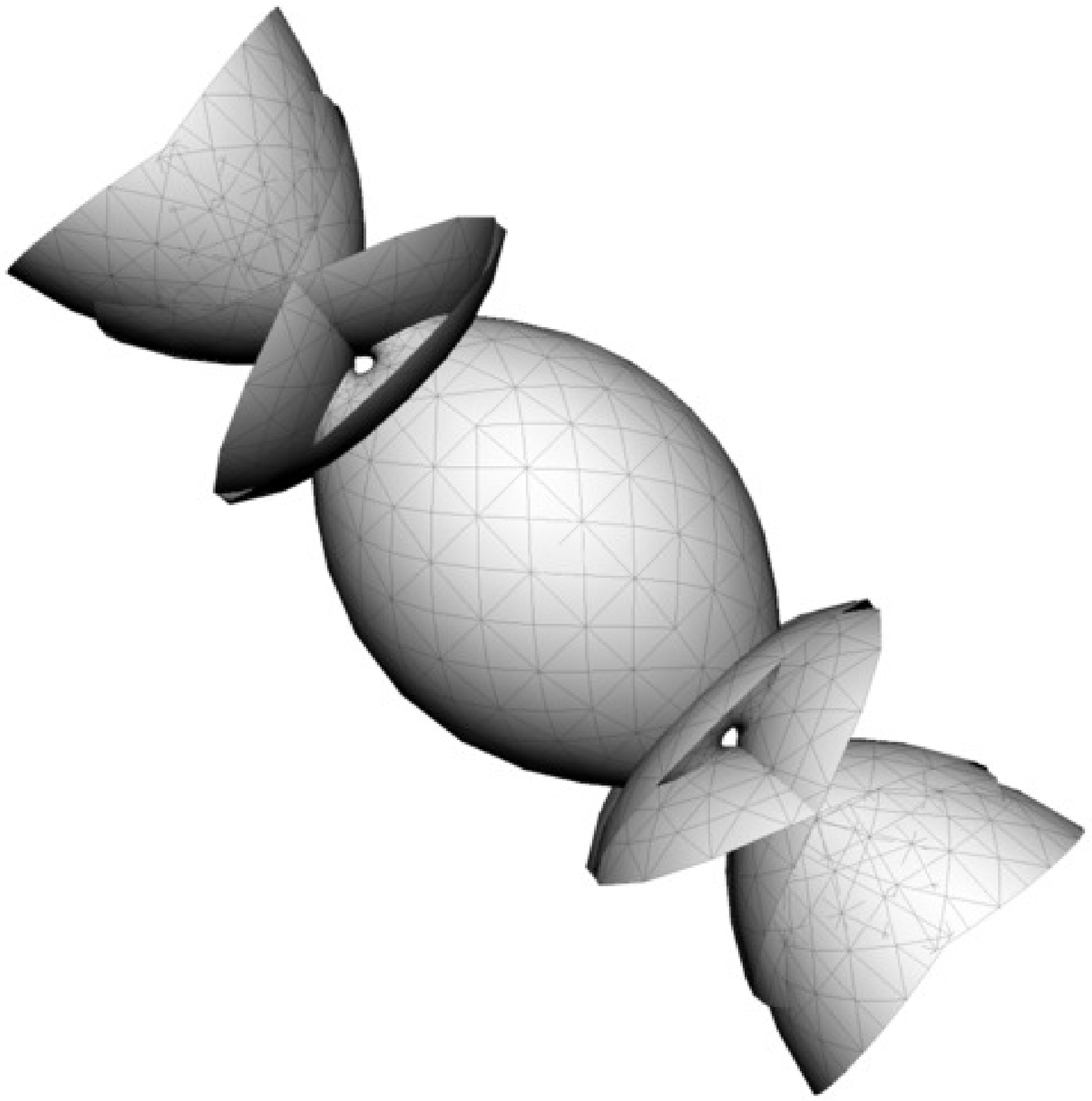}
\includegraphics[width=0.32\textwidth]{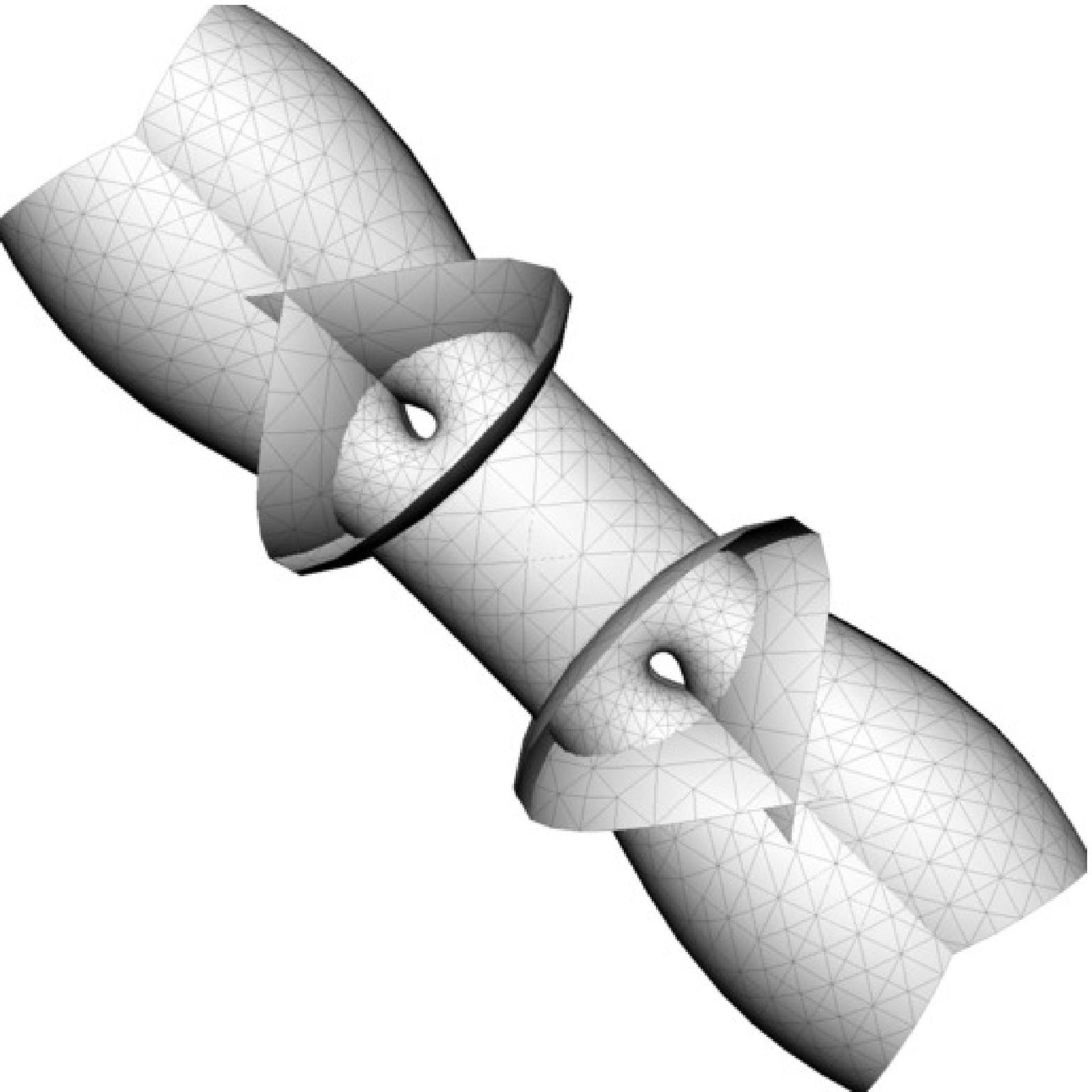}
\includegraphics[width=0.32\textwidth]{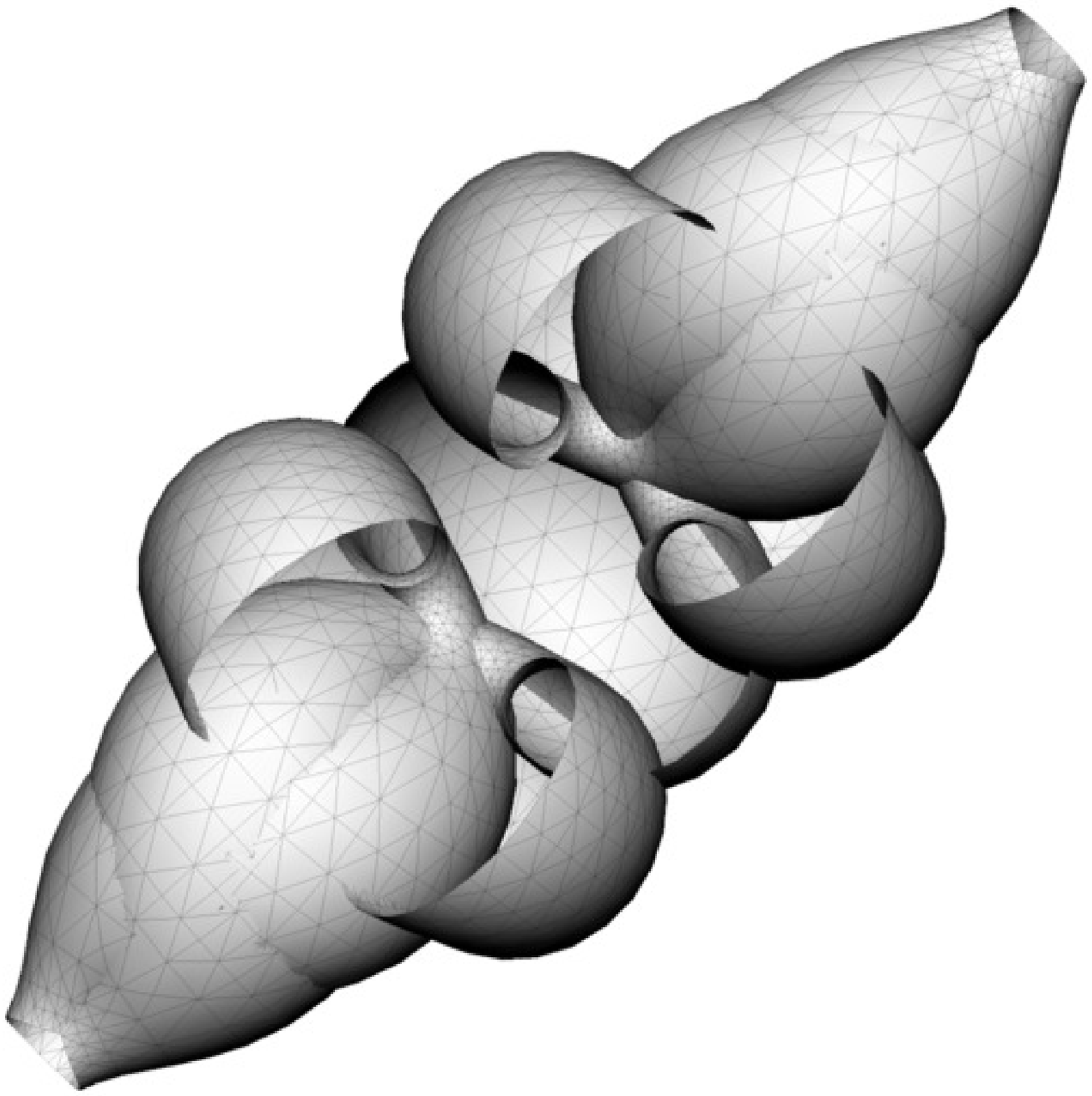}
\includegraphics[width=0.32\textwidth]{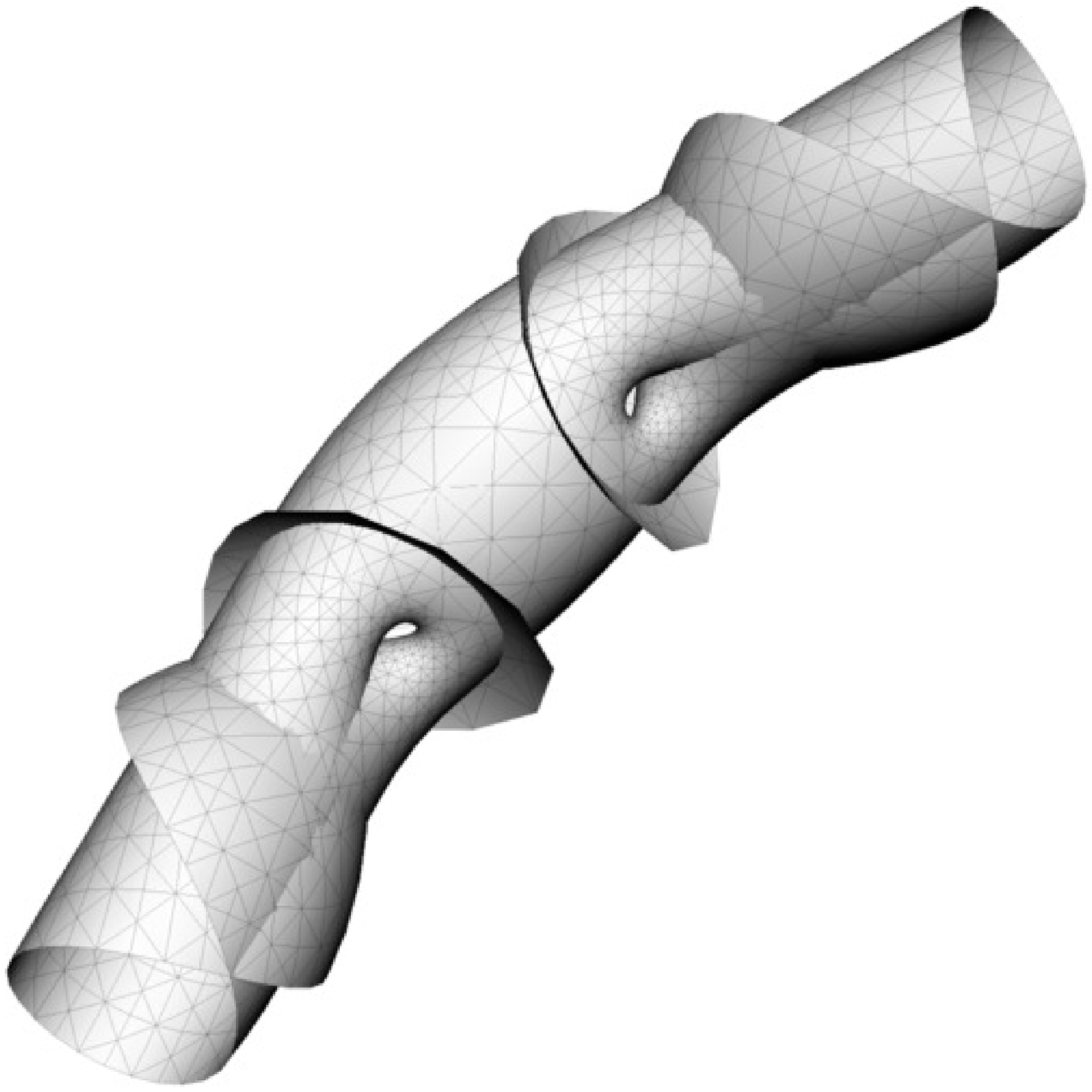}
\includegraphics[width=0.32\textwidth]{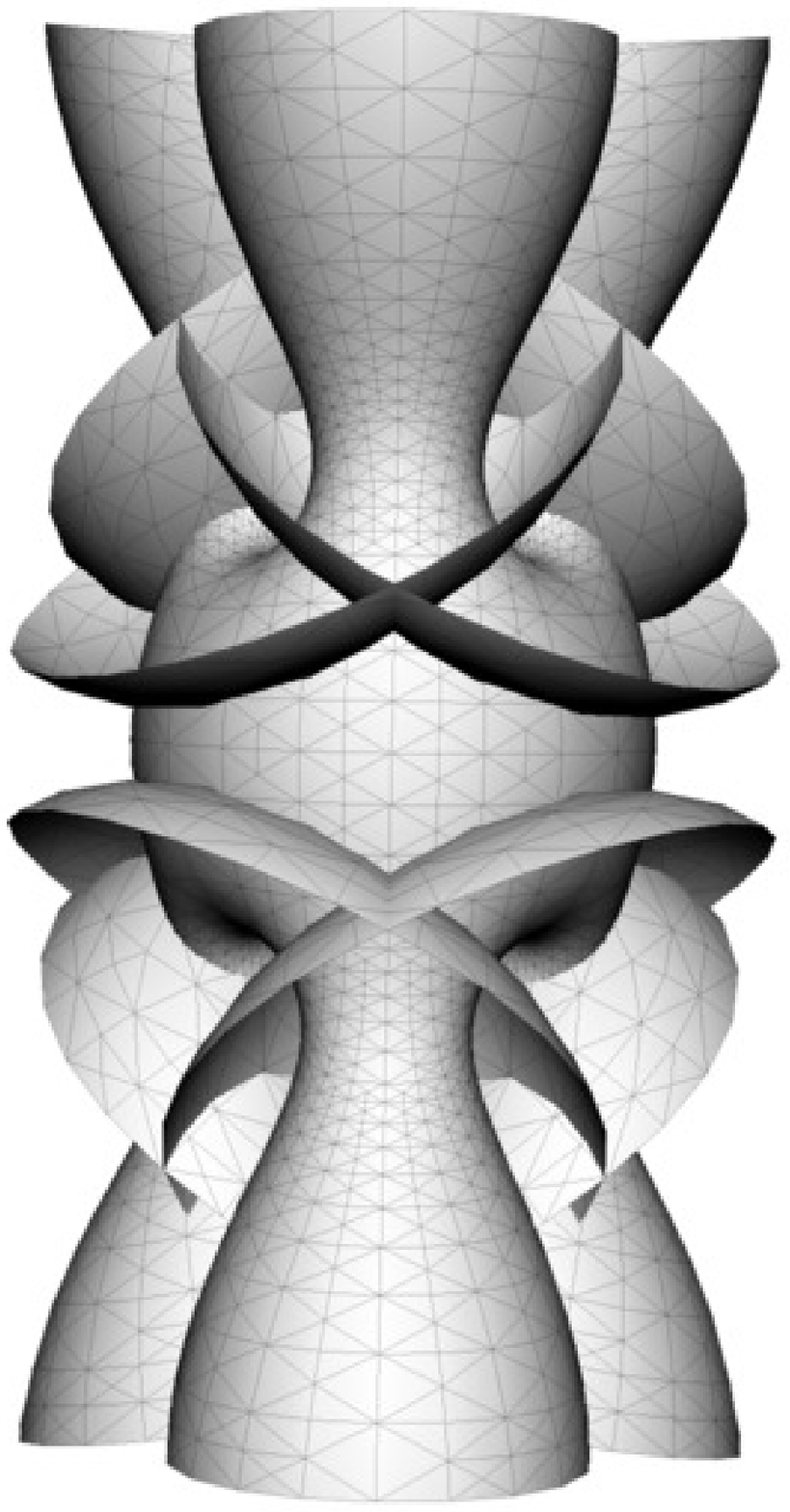}
\includegraphics[width=0.32\textwidth]{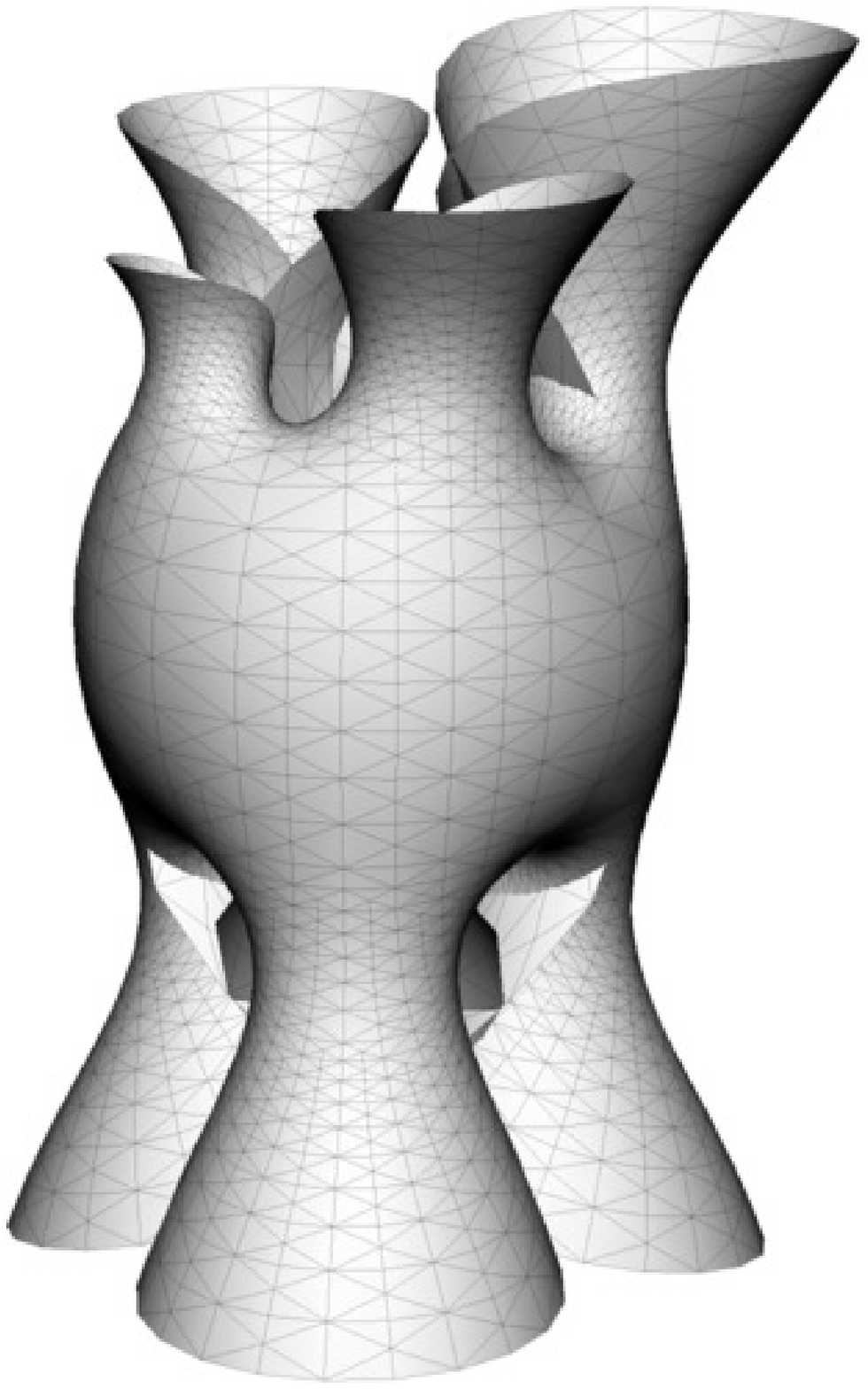}
\caption{Parts of {\sc{cmc}} cylinders with two irregular ends. Each cylinder can be thought of as
a surface with a Delaunay centerpiece of one period
and two irregular ends. In the first row, $2$-legged irregular surfaces emerge from a Delaunay unduloid, cylinder and nodoid. The third surface is cut away to show the internal nodoidal structure. In the second row,
the first two surfaces have $1$-legged and $3$-legged irregular ends respectively. For the first five, $f(z) = a + b(z^n+z^{-n})$ with
$(n,\,a,\,b)$ respectively
\[
(2,\,1/32,\,1/1000),
\ %
(2,\,1/16,\,1/100),
\ %
(2,\,-4/32,\,1/100),
\ %
(1,\,1/32,\,1/50),
\ %
(3,\,1/32,\,1/50).
\qquad
\qquad
\qquad
\qquad
\]
The exceptional last surface with $f(z) = 1/32 + (1/50)(z^3 + z^{-3} + z^4)$, is asymmetric, with different leg counts 3 and 4.
}
\label{fig:cylinder}
\end{figure}

\section{The generalized Weierstra{\ss}~representation}

Our conventions~\cite{Schmitt_Kilian_Kobayashi_Rossman_2007} for the generalized Weierstrass representation~\cite{Dorfmeister_Pedit_Wu_1998} of
{\sc{cmc}} surfaces in $\bbR^3$ are as follows:

1. On a Riemann surface $\Sigma$, let $\xi$ be a holomorphic 1-form with values in the loop algebra
of smooth maps $\bbS^1\to\matsl{2}{\bbC}$. Such 1-forms are called potentials. A potential $\xi$ has to have a simple pole in its upper right entry in the loop parameter $\lambda$ at $\lambda=0$, and has no other poles in the open disk
\begin{equation}
    \calD = \{ \lambda \in \bbC \mid |\lambda |<1 \}\,.
\end{equation}
Moreover, the coefficient of $\lambda^{-1}$ in the upper-right entry of $\xi$ is non-zero on $\Sigma$. Let $\Phi$ be a solution to the ordinary differential equation $\deriv\Phi = \Phi \xi$. Then $\Phi$ is a holomorphic map on the universal cover of $\Sigma$ with values in the loop group of $\matSL{2}{\bbC}$.

2. Let $\Phi = F B$ be the pointwise Iwasawa decomposition on the universal cover: $F$ is \emph{unitary}, that is, at every point of the universal cover a map $\bbS^1\to \matSU{2}{}$, and $B$ is \emph{positive}, that is, at every point of the universal cover it extends holomorphically to $\calD$. Moreover, $B$ is normalized so that $B(\lambda=0)$ is upper-triangular with real positive diagonal elements.

3. Then $\psi = F' F^{-1}$ is an associated family of {\sc{cmc}} immersions of the universal cover into $\matsu{2}{}\cong\bbR^3$. The prime denotes differentiation with respect to $\theta$, where $\lambda=e^{i\theta}$.

The gauge action of positive $g$ is
$\gauge{\xi}{g} \coloneq g^{-1} \xi g + g^{-1} \deriv g$.
If $\Phi$ satisfies $\deriv\Phi = \Phi\xi$,
then $\Psi \coloneq\Phi g$ satisfies $\deriv \Psi = \Psi (\gauge{\xi}{g})$ and induces the same associated family as $\Phi$.

Let $\Phi$ be a solution of $\deriv \Phi = \Phi \xi$,
and let $\tau$ be a deck transformation of the universal cover. The period problem cannot be solved simultaneously for the whole associated family, so we contend ourselves with solving it for the member of the associated family for $\lambda =1$: A sufficient condition that $\left.\psi\right|_{\lambda =1}$ is closed with respect to $\tau$ is that the monodromy $M=(\tau^\ast\Phi)\, \Phi^{-1}$ of $\Phi$ is unitary and satisfies the closing conditions
\begin{equation}
\label{eq:closing}
    \left.M\right|_{\lambda=1} = \pm\id
    \AND
    \tfrac{\deriv}{\deriv\lambda}\left.M\right|_{\lambda=1} = 0
\spaceperiod
\end{equation}
%
\section{The cylinder potential}

We first present the potentials which construct new families of {\sc{cmc}} cylinders with umbilics. The monodromies satisfy the closing conditions, but their unitarity is not automatically satisfied. By imposing symmetries on the potential we can ensure the existence of a unitarizer. The potential for our cylinders on
the twice punctured Riemann sphere
$\CPone\setminus\{0,\,\infty\}$ is of the form

\begin{equation}
\label{eq:cylinder-potential}
\xi
=
\begin{pmatrix}
0 & \lambda^{-1} \\
\fourth \lambda + {(1-\lambda)}^2 f(z) & 0
\end{pmatrix}
\frac{\deriv z}{z}
\end{equation}
where $f:\Cstar \to \bbC$ is an arbitrary holomorphic function with Laurent series
\begin{equation}
\label{eq:laurent-f}
f(z) \eqcolon \sum_{k=-\infty}^\infty a_k z^k
\spaceperiod
\end{equation}
The scalar second order differential equation associated to $\xi$ is Hill's equation \cite{MagW}. The cylinder potential can be viewed as a superposition of an underlying Delaunay potential with two potentials having irregular singularities
\begin{equation*}
\label{eq:superposition}
\xi =
 \begin{pmatrix}0 & 0\\
  {(1-\lambda)}^2 f_-(z) & 0\end{pmatrix}\frac{\deriv z}{z}
+
 \begin{pmatrix}0 & \lambda^{-1}\\
  \fourth \lambda + {(1-\lambda)}^2 f_0 & 0\end{pmatrix}\frac{\deriv z}{z}
+
 \begin{pmatrix}0 & 0\\
  {(1-\lambda)}^2 f_+(z) & 0\end{pmatrix}\frac{\deriv z}{z}
\spacecomma
\end{equation*}
where $f=f_- + f_0 + f_+$ is the decomposition of
$f$ into negative, constant and positive Fourier modes.
The middle term is a Delaunay potential. The umbilic points are located at the zeroes of the function $f$.

\section{The diagonal unitarizer}

The map $\Phi$ satisfying the initial value problem
$\deriv \Phi = \Phi\xi$ and $\Phi(1)=\id$
has monodromy $M$ along the circle $\abs{z}=1$.
Only if $M$ is unitary and satisfies the closing conditions can we conclude that the resulting {\sc{cmc}} immersion
closes after one traversal of the circle.

A smooth unitarizer on $\bbS^1$ can be constructed
once the monodromy is shown to be pointwise unitarizable on $\bbS^1$. This is the content of Proposition~\ref{prop:diagonal-unitarizer-X}
and is in the spirit
of~\cite{Kilian_Kobayashi_Rossmann_Schmitt_2005, Rossman_Schmitt_2006, Schmitt_Kilian_Kobayashi_Rossman_2007}.
For cylinders, the monodromy unitarizer is not unique.
Symmetries on the potential can be imposed to ensure the
existence of a diagonal unitarizer. The benefit of a diagonal unitarizer is that it is easiest to construct, and that it induces symmetries on the surface.

A matrix $M\in\matSL{2}{\bbC}\setminus\{\pm\id\}$ can be conjugated to $\matSU{2}{}$ if and only if $\trace M\in(-2,\,2)$. Let $\Delta\subset\matSL{2}{\bbC}$ denote the subgroup of diagonal elements, which we sometimes write as $\mathrm{diag}(x,\,1/x)$. An element $M\in\matSL{2}{\bbC}\setminus\{\pm\id\}$
is unitarizable by an element of $\Delta$ if and only
it is unitarizable, its diagonal elements are complex conjugates of each other, and their product is less than $1$.

A map $M:\calC_r\to\matSL{2}{\bbC}$ from the circle $\calC_r$ of radius $r$ is \emph{$r$-unitary}
if it extends holomorphically to $\bbS^1$ and is unitary there. The map $M$ is \emph{$r$-unitarizable} if it can
be conjugated to an $r$-unitary map.
%

\begin{proposition}
\label{prop:diagonal-unitarizer-X}
If $M:\bbS^1\to\matSL{2}{\bbC}$ is pointwise unitarizable
by elements of $\Delta$, except at finitely many points on $\bbS^1$, then there exists a holomorphic map $V:\calD\to\Delta$ which $r$-unitarizes $M$ for every $r\in(0,\,1)$. The map $V$ is unique up to left multiplication by an element of $\Delta\cap\matSU{2}{}$.
\end{proposition}
\begin{proof}
Let us write
\begin{equation}
    M = \begin{pmatrix}a & b\\c & d
    \end{pmatrix} \spaceperiod
\end{equation}
By Lemma 9 in~\cite{Schmitt_Kilian_Kobayashi_Rossman_2007},
there exist nonzero holomorphic maps $p,\,q:\calD\to\Cstar$
\begin{equation}
    p p^\dagger =  c c^\dagger
    \AND
    q q^\dagger =  - b c \spacecomma
\end{equation}
where $p,\,q$ extend analytically to $\bbS^1$, and $p^\dagger(\lambda) \coloneq \ol{p(1/\ol{\lambda})}$. 
Since $p$ and $q$ are non-zero on $\calD$, they have single-valued square roots there, allowing us to define the map $V:\calD \to \Delta$ by $V = \mathrm{diag}(\sqrt{p/q},\,\sqrt{q/p})$. Then $P = V M V^{-1}$ is unitary on $\bbS^1$ away from the singular set of $V$. It follows by Lemma 10 in~\cite{Schmitt_Kilian_Kobayashi_Rossman_2007} that it extends holomorphically to all of $\bbS^1$, and is unitary there.

To show uniqueness, suppose $V_1$ and $V_2$ are two such unitarizers of $M$. Then $V_2V_1^{-1}$ is a diagonal unitarizer of the unitary map $V_1MV_1^{-1}$. This implies $V_2V_1^{-1}$ is diagonal and unitary. Since $V_1$ and $V_2$ are positive, then $V_1$ and $V_2$ coincide up to a constant phase.
\end{proof}

\section{Monodromy series}

Let $M$ be the monodromy with respect to the curve $\gamma(s) = e^{is},\,s\in[0,\,2\pi]$ of $\Phi$ for the cylinder potential $\xi$ with $\Phi(1)=\id$.
Proposition~\ref{prop:one-pole-monodromy-X}
computes the series expansion of the monodromy
with respect to $\lambda$ at $\lambda=1$,
in terms of the coefficients of the Laurent
series of $f$.
%
\begin{proposition}
\label{prop:one-pole-monodromy-X}
The series expansion of the monodromy
is
\begin{equation}
\label{eq:monodromy-expansion}
M = \id + A {(\lambda-1)}^2 + \Order({(\lambda-1)}^3)
\end{equation}
where
\begin{equation}
\label{eq:unitary1-3.0}
A =
-\frac{\pi i}{2}
  \begin{pmatrix}1 & -1\\\half & \half\end{pmatrix}
  \begin{pmatrix}a_{0} & -a_{-1}\\ a_{1} & -a_{0}\end{pmatrix}
  \begin{pmatrix}1 & -1\\\half & \half\end{pmatrix}^{-1}
\spaceperiod
\end{equation}
\end{proposition}
\begin{proof}
By the gauge
\begin{equation}
    g = \begin{pmatrix} 1/\sqrt{\lambda} & 0 \\ 0 & \sqrt{\lambda} \end{pmatrix}
    \begin{pmatrix}1/\sqrt{z} & 0 \\ 0 & \sqrt{z}\end{pmatrix}
    \begin{pmatrix}1 & 0 \\ 1/(2z) & 1\end{pmatrix}
    \begin{pmatrix}1 & -1 \\ 0 & 1\end{pmatrix}
    \spacecomma
\end{equation}
we obtain
\begin{equation}
    \eta \coloneq \gauge{\xi}{g} =
    \eta_0 + t \eta_1
    \spacecomma\quad
    \eta_0 \coloneq \begin{pmatrix}0 & \alpha \\ 0 & 0\end{pmatrix}
    \spacecomma\quad
    \eta_1 \coloneq \beta\begin{pmatrix}1 & -1 \\ 1 & -1\end{pmatrix}
    \spacecomma
\end{equation}
where
\begin{equation}
    \alpha = \deriv z
    \AND
    \beta = -\frac{4f(z)}{z^2}\deriv z
\end{equation}
and
\begin{equation} \label{eq:t-def}
    t = -\frac{ {(\lambda-1)}^2 }{4\lambda} = \sin^2 \tfrac{\theta}{2}
    \spaceperiod
\end{equation}
Let $\Psi = \Psi_0 + \Psi_1 t + \Order(t^2)$
be the solution to the initial value problem
\begin{equation}
    \deriv \Psi = \Psi\eta
    \spacecomma\quad
    \Psi(1) = \id
    \spacecomma
\end{equation}
and let $P = P_0 + P_1 t + \Order(t^2)$
be the monodromy of $\Psi$. To compute $P_0$ and $P_1$, equate the coefficients of powers of $t$ in
\begin{equation}
    \deriv\Psi_0 + \deriv\Psi_1 t + \Order(t^2) =
    (\Psi_0 + \Psi_1 t + \Order(t^2))(\eta_0 + \eta_1 t)
\end{equation}
to obtain the two equations
\begin{subequations}
    \begin{align}
    \label{eq:unitary1-3.1a-X}
    \deriv \Psi_0  &= \Psi_0\eta_0
    \spacecomma\quad \Psi_0(1) = \id
    \spacecomma\\
    \label{eq:unitary1-3.1b-X}
    \deriv\left( \Psi_1\Psi_0^{-1}\right) &= \Psi_0\eta_1\Psi_0^{-1}
    \spacecomma\quad \Psi_1(1) = 0
    \spaceperiod
    \end{align}
\end{subequations}
The solution to~\eqref{eq:unitary1-3.1a-X} is
\begin{equation}
    \Psi_0 =
    \begin{pmatrix}1 & \int \alpha \\ 0 & 1\end{pmatrix}
    \spacecomma
\end{equation}
where the path integral is along a path based
at $1$. Since $\int\alpha = z-1$, then $\int_\gamma\alpha = 0$ for $\gamma(s) = e^{is},\,s\in[0,\,2\pi]$. Hence $M_0=\id$. Solve~\eqref{eq:unitary1-3.1b-X} by computing
\begin{equation}
    \Psi_1\Psi_0^{-1} =
    \int \Psi_0\eta_1\Psi_0^{-1} =
    \int
    \beta\begin{pmatrix}1+\int\alpha & -{(1+\int\alpha)}^2 \\
    1 & -(1+\int\alpha)
    \end{pmatrix}
    =
    \int \beta\begin{pmatrix}z & -z^2 \\ 1 & -z
    \end{pmatrix}
    \spaceperiod
\end{equation}
By the residue theorem we obtain
\begin{equation}
\label{eq:unitary1-3.2-X}
    P_1 = \left(\int_\gamma\left(\Psi_0\eta_1\Psi_0^{-1}\right)\right)M_0
    =
    2\pi i
    \begin{pmatrix}a_{0} & -a_{-1}\\ a_{1} & -a_{0}
    \end{pmatrix}
\end{equation}
The series for the monodromy $M$ of $\Phi$ now follows from $M = g(1)Pg^{-1}(1)$.
\end{proof}


The following lemma is used to show that
if the monodromy satisfies the closing conditions,
then the monodromy after unitarization also does,
even if the unitarizer is singular on $\bbS^1$.
It also computes the weight associated to
an end of a {\sc{cmc}} immersion from the monodromy
before unitarization.

\begin{lemma}
\label{lem:unitary-monodromy-series}
Let $M:(-\epsilon,\,\epsilon)\to\matSL{2}{\bbC}$
and $U:(-\epsilon,\,\epsilon)\to\matSU{2}{}$
be analytic maps with equal traces.
If $M = \id + M_2 t^2 + \Order(t^3)$
then
$U = \id + U_2 t^2 + \Order(t^3)$
and $U_2^2 = M_2^2$.
\end{lemma}
\begin{proof}
Let $\tau = \half\trace M = \half\trace U$.
Then $\tau\id = \half(M + M^{-1}) = \half(U + U^{-1})$.
Let
\begin{equation}
M = \sum_{j=0}^\infty M_j t^j
\AND
U = \sum_{j=0}^\infty U_j t^j
\spaceperiod
\end{equation}
Then
\begin{equation}
    M^{-1} = \id - M_2 t^2 -M_3 t^3 + (M_2^2-M_4)t^4 + \Order(t^5)
    \spacecomma
\end{equation}
and therefore 
\begin{equation}
    \half(M + M^{-1}) = \id + \half M_2^2 t^4 + \Order(t^5)
    \spaceperiod
\end{equation}
Looking at $U$ we have
\begin{equation}
    U^{-1} = \id - U_1 t + (U_1^2 - U_2) t^2 + \Order(t^3)
    \spacecomma
\end{equation}
and thus
\begin{equation}
    \half(U + U^{-1}) = \id + \half U_1^2 t^2 + \Order(t^3)
    \spaceperiod
\end{equation}
Comparing with the series for $\half(M+M^{-1})$,
then $U_1^2 = 0$. Since $U_1\in\matsu{2}{}$, then $U_1 = 0$. From
\begin{equation}
    U^{-1} = \id - U_2 t^2 - U_2 t^2 - U_3 t^3 + (U_2^2 - U_4) t^4 + \Order(t^5)
    \spacecomma
\end{equation}
we obtain
\begin{equation}
    \half(U + U^{-1}) = \id + \half U_2^2 t^4 + \Order(t^5)
    \spaceperiod
\end{equation}
Comparing with the series for $\half(M+M^{-1})$, we conclude that $U_2^2 = M_2^2$.
\end{proof}

\section{Cylinder construction}

With $\rho(z)=\ol{z}$ and $\sigma(z)= 1/\ol{z}$,
we impose the symmetries
\begin{equation} \label{eq:f-symmetries}
    f = \ol{\rho^\ast f} \AND
    f = \ol{\sigma^\ast f} \spaceperiod
\end{equation}
These ensure that the trace of the monodromy is real on $\bbS^1$, and that there exists a diagonal unitarizer. From the series expansion~\eqref{eq:monodromy-expansion}, the trace is $2$ at $\lambda=1$.
The inequality $a_0^2 - a_{-1}a_{1} > 0$ guarantees that the trace along $\bbS^1$ is decreasing from $2$ at $\lambda=1$.
By a technique in~\cite{Kilian_Kobayashi_Rossmann_Schmitt_2005},
a rescaling of the potential
\begin{equation}
\label{eq:cylinder-potential-2}
    \xi_{\scale} =
    \begin{pmatrix}
    0 & \lambda^{-1} \\
    \fourth \lambda + {(1-\lambda)}^2 \scale f(z) & 0
    \end{pmatrix}
    \frac{\deriv z}{z}
    \spacecomma
\end{equation}
with $\scale>0$, ensures that the monodromy trace remains in the interval $[-2,\,2]$ for all $\lambda$ on $\bbS^1$.
Proposition~\ref{prop:diagonal-unitarizer-X} constructs
a smooth map which $r$-unitarizes the monodromy.
This unitarizer as initial condition yields
a family of {\sc{cmc}} cylinders with umbilics.
Because the unitarizer is diagonal,
the symmetries of the potentials descend to
ambient symmetries of the cylinders.

\begin{theorem}
\label{thm:cylinder-construction}
Assume $f$ has the symmetries \eqref{eq:f-symmetries} and $\kappa\coloneq a_0^2 - a_{-1}a_{1} > 0$. For sufficiently small $\scale>0$, the potential $\xi_{\scale}$ gives rise to a {\sc{cmc}} cylinder with umbilics at the roots of $f$, and two symmetry planes.
\end{theorem}

\begin{proof}
Let $\Phi_{\scale}$ satisfy
$\deriv \Phi_{\scale} = \Phi_{\scale}\xi_{\scale}$ and $\Phi_{\scale}(1,\,\lambda) = \id$,
and let $M_{\scale}$ be the monodromy of $\Phi_{\scale}$ along $\abs{z}=1$.
Due to the form of $\xi_{\scale}$, $M_{\scale}$ satisfies the closing conditions
$\left.M\right|_{\lambda=1} = \id$ and
$\frac{\deriv}{\deriv\lambda}\left.M\right|_{\lambda=1} = 0$.

Let $\Lambda = \diag(\sqrt{\lambda},\,1/\sqrt{\lambda})$. The symmetries on $f$ imply the symmetries on the gauged potential
\begin{equation}
    \eta = \Lambda^{-1} \ol{\rho^\ast \eta(1/\ol{\lambda})} \,\Lambda
    \AND
    \eta = \Lambda^{-1} h\, \ol{\sigma^\ast \eta(1/\ol{\lambda})} \,h^{-1} \Lambda
\end{equation}
where $h=\diag(i,\,-i)$. This gives the symmetries
\begin{equation}
    \Phi = R \ol{\rho^\ast \Phi} \Lambda
    \AND
    \Phi = S \ol{\sigma^\ast \Phi}h^{-1} \Lambda
\end{equation}
for some $z$-independent $R$, $S$.
Evaluation at the fixed point $z=1$ gives
$R=\Lambda^{-1}$  and $S=\Lambda^{-1}h$.
Hence the monodromy has the symmetries
\begin{equation}
    M_{\scale}(\lambda) =
    \Lambda^{-1} {\ol{M_{\scale}( {1/\ol{\lambda}})}}^{-1} \Lambda
    \AND
    M_{\scale}(\lambda) =
    \Lambda^{-1} h\ol{M_{\scale}( {1/\ol{\lambda}})}h^{-1} \Lambda \spaceperiod
\end{equation}
This implies that the diagonal elements of $M_{\scale}$ are equal, and real on $\bbS^1$.

Let $M=M_1$, and $t$ as in \eqref{eq:t-def}. Due to the assumption $\kappa>0$, we have that $\trace M$ is decreasing in $t$ at $t=0$ from $2$.
Hence there exists ${\scale}_0\in(0,\,1]$ such that $\trace M \le 2$ for $t\in[0,\,{\scale}_0]$.
It follows that $\trace M_{{\scale}_0}\in[-2,\,2]$ for $\lambda\in\bbS^1$, and is therefore unitarizable there.
For similar uses of this technique see~\cite{Kilian_Kobayashi_Rossmann_Schmitt_2005,Dorfmeister-Haak-2003}.

The product of the diagonal elements of $M$ is less than $1$. Hence by Proposition~\ref{prop:diagonal-unitarizer-X}
there exists a map $V:\calD\to\Delta$ such that $VM_{{\scale}_0}V^{-1}$ is unitary on $\bbS^1$.

By Lemma 13 in~\cite{Schmitt_Kilian_Kobayashi_Rossman_2007},
the monodromy
$VMV^{-1}$ of $V\Phi_{{\scale}_0}$ extends to a holomorphic map $\bbS^1\to\matSU{2}{}$. By Lemma~\ref{lem:unitary-monodromy-series}, it satisfies
the closing conditions~\eqref{eq:closing}. Hence the induced {\sc{cmc}} immersion closes.
The proof of the symmetries statement is deferred to the next section.
\end{proof}
The \emph{force} associated to an element in the fundamental group~\cite{Korevaar_Kusner_Solomon_1989,Bobkeno_1994} is the matrix $A\in\matsu{2}{}$ in the series expansion of the monodromy
\begin{equation}
    M = \id + A \theta^2 + \Order(\theta^3)
\end{equation}
where $\lambda = e^{i\theta}$.
The force is a homomorphism from the fundamental group to
$\matsu{2}{} \cong \bbR^3$. Its length $\abs{A} = \sqrt{\det A}$ is the \emph{weight}.
By Proposition~\ref{prop:one-pole-monodromy-X},
and Lemma~\ref{lem:unitary-monodromy-series},
the weight of each of the {\sc{cmc}} cylinders
constructed in Theorem~\ref{thm:cylinder-construction}
is given by $\frac{\pi}{2}\sqrt{a_0^2 - a_{-1}a_{1}}$,
where $a_k$ are the Laurent coefficients of $\scale f$.

Graphics suggest that the ends of these cylinders
might be asymptotic to Smyth surfaces. However, this is not the case because our cylinders in general have nonvanishing end weight, while the end weights of {\sc{cmc}} planes always vanish.

\section{Symmetry}

It remains to prove the symmetry statement in Theorem~\ref{thm:cylinder-construction}.
We will show how the symmetries imposed on $f$ descend to symmetries on the induced {\sc{cmc}} immersion.

\begin{theorem}
All cylinders constructed in Theorem~\ref{thm:cylinder-construction}
have two perpendicular planes of reflective symmetry.
\end{theorem}

\begin{proof}
Let $\Psi$ be the solution of $\deriv \Psi = \Psi\xi,\,\Psi(1) = V$, with $V$ diagonal such that $\Psi$ has unitary monodromy. Let $\Psi = F B$ be the Iwasawa factorization of $\Psi$. With $\mu(z) \coloneq 1/z$ the symmetry $f = \ol{\mu^\ast f}$ induces the symmetry
\begin{equation}
    \mu^\ast \xi = h\xi h^{-1}
    \spacecomma
\end{equation}
where $h \coloneq \diag(i,\,-i)$.
Hence $\Psi$ has the symmetry
$\mu^\ast\Psi = h \Psi h^{-1}$ for some map $R$.
Evaluating at the fixed point $z=1$ of $\mu$ yields
$R=h$, which is unitary. Hence
\begin{equation}
    \mu^\ast F \cdot \mu^\ast B =
    \mu^\ast\Psi = h \Psi h = h F \cdot B h \spaceperiod
\end{equation}
Identifying unitary and positive parts implies that
\begin{equation}
    \ol{\mu^\ast F} = h F \delta
    \spacecomma
\end{equation}
where $\delta\in\matSU{2}{}$ is $\lambda$-independent.
Passing to the immersion $\psi=F'F^{-1}$, we have that
\begin{equation}
    \mu^\ast \psi = h\psi h^{-1} \spacecomma
\end{equation}
which is an ambient rotation by $\pi$ which exchanges the ends.

For the orientation-reversing symmetry $\sigma(z) \coloneq \ol{z}$, the symmetry $f=\ol{\sigma^\ast f}$
induces on the potential $\xi$ the symmetry
\begin{equation}
    \xi = \Lambda^{-1}\ol{\sigma^\ast\xi(1/\ol{\lambda})}\Lambda
    \spacecomma
\end{equation}
where $\Lambda = \diag(\sqrt{\lambda},\,1/\sqrt{\lambda})$. Hence $\Psi$ has the symmetry
\begin{equation}
    R(\lambda) \Psi(\lambda) =
    \ol{\sigma^\ast\Psi(1/\ol{\lambda})}\Lambda
\end{equation}
for some $z$-independent $R$. Evaluation at the fixed point $1$ of $\sigma$ yields $R(\lambda) = \ol{V(1/\ol{\lambda})} \Lambda(\lambda) V^{-1}(\lambda)$. Since $V$ is diagonal, then $R$ is a unitary.

In the notation below the transform $\lambda\mapsto 1/\ol{\lambda}$ is omitted. Then
\begin{equation}
    \ol{\sigma^\ast F} \cdot \ol{\sigma^\ast B} =
    \ol{ \sigma^\ast\Psi} = R \Psi = R F \cdot B
    \spaceperiod
\end{equation}
Identifying unitary and positive parts gives
\begin{equation}
    \sigma^\ast F = R F \delta
    \spacecomma
\end{equation}
where $\delta\in\matSU{2}{}$ is $\lambda$-independent.
Passing to the immersion $\psi=F'F^{-1}$ yields the orientation-reversing symmetry
\begin{equation}
    \ol{\sigma^\ast \psi} = -R \psi R^{-1} - R' R^{-1}
    \spaceperiod
\end{equation}
Since $\sigma$ is an involution on the universal cover,
then this symmetry is an involution, and hence
a reflection in a plane which fixes each end.

The composition $\sigma\circ\mu=\mu\circ\sigma$
is a reflection in a plane which exchanges the ends.
Since the Klein four-group generated by $\rho$ and $\sigma$ is abelian, the planes are perpendicular and the axis of the rotation is the intersection of the two planes.
\end{proof}


\bibliographystyle{amsplain}
\bibliography{references}


\end{document}